\documentclass{svproc}
\usepackage{url}
\usepackage{comment}
\usepackage{siunitx}
\usepackage{relsize}
\usepackage{ifthen}
\usepackage[colorinlistoftodos]{todonotes}

\usepackage[caption=false]{subfig}

\usepackage[vlined,ruled,linesnumbered]{algorithm2e}
\usepackage{graphics} % for pdf, bitmapped graphics files
\usepackage{rotating}
\usepackage{color}
\usepackage{enumerate}
\usepackage[T1]{fontenc}
\usepackage{psfrag}
\usepackage{epsfig} % for postscript graphics files
\usepackage{booktabs}
\usepackage{graphicx,url}
\usepackage{multirow}
\usepackage{array}
\usepackage{latexsym}
\usepackage{amsfonts}
\usepackage{amsmath}
\usepackage{amssymb}
\usepackage{mathtools}
\usepackage{xstring}
\usepackage[noend]{algorithmic}
\usepackage{multirow}
\usepackage{xcolor}
\usepackage{prettyref}
\usepackage{flexisym}
\usepackage{bigdelim}
\usepackage{breqn} % load this last
\usepackage{listings}

\usepackage{enumitem}
\usepackage{xspace}
\usepackage{bm}
\graphicspath{{./figures/}}
\usepackage{tikz}
\usetikzlibrary{matrix,calc}

\usepackage{mdwlist}

\makecompactlist{itemize}{stditemize}

\newtheorem{assumption}[theorem]{Assumption}
\newcommand{\cf}{\emph{cf.}\xspace}

\newcommand{\bdmath}{\begin{dmath}}
\newcommand{\edmath}{\end{dmath}}
\newcommand{\beq}{\begin{equation}}
\newcommand{\eeq}{\end{equation}}
\newcommand{\bdm}{\begin{displaymath}}
\newcommand{\edm}{\end{displaymath}}
\newcommand{\bea}{\begin{eqnarray}}
\newcommand{\eea}{\end{eqnarray}}
\newcommand{\beal}{\beq \begin{array}{ll}}
\newcommand{\eeal}{\end{array} \eeq}
\newcommand{\beas}{\begin{eqnarray*}}
\newcommand{\eeas}{\end{eqnarray*}}
\newcommand{\ba}{\begin{array}}
\newcommand{\ea}{\end{array}}
\newcommand{\bit}{\begin{itemize}}
\newcommand{\eit}{\end{itemize}}
\newcommand{\ben}{\begin{enumerate}}
\newcommand{\een}{\end{enumerate}}

\newcommand{\calA}{{\cal A}}

\newcommand{\calC}{{\cal C}}

\newcommand{\eg}{\emph{e.g.,}\xspace}
\newcommand{\ie}{\emph{i.e.,}\xspace}

\newcommand{\hide}[1]{}
\newcommand{\wrt}{w.r.t.\xspace}

\newcommand{\grayout}[1]{{\color{gray} #1}}

\newcommand{\hiddenText}{{\color{gray} hidden text.}}
\newcommand{\hideWithText}[1]{\hiddenText}

\newcommand{\subject}{\text{ subject to }}
\DeclareMathOperator*{\argmax}{arg\,max}

\newcommand{\norm}[1]{\left\| #1 \right\|}

\newcommand{\tran}{^{\mathsf{T}}}

\newcommand{\trace}[1]{\mathrm{tr}\left(#1\right)}

\newcommand{\rank}[1]{\mathrm{rank}\left(#1\right)}

\newcommand{\inv}{^{-1}}

\newcommand{\eye}{{\mathbf I}}

\newcommand{\Real}[1]{ { {\mathbb R}^{#1} } }

\newcommand{\blue}[1]{{\color{blue}#1}}

\newcommand{\linkToPdf}[1]{\href{#1}{\blue{(pdf)}}}
\newcommand{\linkToPpt}[1]{\href{#1}{\blue{(ppt)}}}
\newcommand{\linkToCode}[1]{\href{#1}{\blue{(code)}}}
\newcommand{\linkToWeb}[1]{\href{#1}{\blue{(web)}}}
\newcommand{\linkToVideo}[1]{\href{#1}{\blue{(video)}}}
\newcommand{\linkToMedia}[1]{\href{#1}{\blue{(media)}}}
\newcommand{\award}[1]{\xspace} % {{\red{#1}}} % omit awards

\usepackage[pagebackref=true,breaklinks=true,letterpaper=true,colorlinks,bookmarks=true]{hyperref}
\renewcommand{\norm}[1]{\left\lVert #1 \right\rVert}
\newcommand{\inprod}[2]{\left\langle #1, #2 \right\rangle}

\newcommand{\mymid}{\ \middle\vert\ }
\newcommand{\cbrace}[1]{\left\{#1\right\}}

\newcommand{\bmat}{\left[ \begin{array}}
\newcommand{\emat}{\end{array}\right]}

\newcommand{\parentheses}[1]{\left(#1\right)}

\newcommand{\tldV}{\widetilde{V}}
\newcommand{\hatV}{\widehat{V}}
\newcommand{\tldTheta}{\widetilde{\Theta}}
\newcommand{\tldl}{\widetilde{l}}

\newcommand{\bracket}[1]{\left[#1\right]}
\newcommand{\abs}[1]{\left|#1\right|}
\newcommand{\dx}{\dot{x}}
\newcommand{\bbU}{\mathbb{U}}

\newcommand{\bbN}{\mathbb{N}}

\newcommand{\poly}[1]{\mathbb{R}[#1]}

\newcommand{\ceil}[1]{\left\lceil #1 \right\rceil}

\newcommand{\bbK}{\mathbb{K}}

\newcommand{\HYedit}[1]{#1}

\newcommand{\snew}{\mathfrak{s}}
\newcommand{\cnew}{\mathfrak{c}}

\newcommand{\thetalb}{\underline{\theta}}
\newcommand{\thetaub}{\bar{\theta}}
\newcommand{\volume}[1]{\text{vol}(#1)}
\newcommand{\tldbeta}{\tilde{\beta}}

\newcommand{\kscedit}[1]{#1}

\newcommand{\qedsymbol}{$\blacksquare$}

\begin{document}
\mainmatter              % start of a contribution
\title{Verification and Synthesis of Robust Control Barrier Functions: Multilevel Polynomial Optimization and Semidefinite Relaxation}

\titlerunning{Verification and Synthesis of Robust CBF}  % abbreviated title (for running head)
\author{Shucheng Kang\inst{1} \and Yuxiao Chen\inst{2} \and Heng Yang\inst{1,2} \and Marco Pavone\inst{2}}
\authorrunning{ } % abbreviated author list (for running head)
\tocauthor{Ivar Ekeland, Roger Temam, Jeffrey Dean, David Grove,
Craig Chambers, Kim B. Bruce, and Elisa Bertino}
\institute{\kscedit{School of Engineering and Applied Sciences, Harvard University, USA
\and
NVIDIA Research, USA}}

\maketitle              % typeset the title of the contribution

%%%%%%%%%%%%%%%%%%%%%%%%%%%%%%%%%%%%%%%%%%%%%%%%%%%%%%%%%%%%%%%%%%%%%%%%%%%%%%%%
%!TEX root = ../main.tex
\vspace{-4mm}
\begin{abstract}
    We study \HYedit{the problem of} verification and synthesis of \emph{robust} control barrier functions (CBF) for control-affine polynomial systems with \HYedit{bounded} additive uncertainty and \HYedit{convex polynomial constraints on the control}. We first formulate robust CBF verification and synthesis as \emph{multilevel} polynomial optimization problems (POP), where verification  optimizes --in three levels-- the uncertainty, control, and state, while synthesis additionally optimizes the parameter of a chosen parametric CBF candidate. We then show that, by invoking the KKT conditions of the inner optimizations over uncertainty and control, the verification problem can be simplified as a \emph{single-level} POP and the synthesis problem reduces to a \emph{min-max} POP. This reduction leads to multilevel semidefinite relaxations. For the verification problem, we apply Lasserre's hierarchy of moment relaxations. For the synthesis problem, \HYedit{we draw connections to existing relaxation techniques for \emph{robust min-max} POP}, which first use sum-of-squares programming to find \emph{increasingly tight} polynomial lower bounds to the unknown \emph{value function} of the verification POP, and then call Lasserre's hierarchy again to maximize the lower bounds. Both semidefinite relaxations guarantee \emph{asymptotic global convergence} to optimality. We provide an in-depth study of our framework on the controlled Van der Pol Oscillator, both with and without additive uncertainty.


\end{abstract}

%%%%%%%%%%%%%%%%%%%%%%%%%%%%%%%%%%%%%%%%%%%%%%%%%%%%%%%%%%%%%%%%%%%%%%%%%%%%%%%%
%!TEX root = ../main.tex

\section{Introduction}
Safety is critical in high-integrity autonomous systems. In the control community, algorithms based on \emph{energy functions} are widely employed to constrain the system state in a safe set. Among them, \emph{control barrier functions} (CBF) \cite{ames2014cdc-cbforigin} have attracted increasing research interests. A CBF encodes the safe set as its zero superlevel set that is \emph{control invariant}, \ie if the system starts within the safe set, a control sequence exists to keep the system in the safe set. Because dynamics \emph{uncertainty} is ubiquitous in real-world systems, the notion of a \emph{robust} CBF emerges to ensure the system stays safe despite the existence of uncertainties.
% Whenever the CBF's value approaches zero, a control input is picked to guarantee that the function's time derivative is non-negative. 

% On the other hand, dynamics uncertainty is prevalent in real-world applications. With uncertain dynamics models, 
Current literature around (robust) CBF mainly focus on the CBF \emph{deployment} problem: given a CBF, synthesize a safe controller with quadratic programming \cite{taylor2020acc-robustqp} or second-order cone programming \cite{buch21csl-robust}. However, two challenges remain largely unsolved: how to verify and synthesize robust CBFs?
\begin{enumerate}[label=(\roman*)]
    \item Verification. Given a robust CBF candidate, for \emph{all states} belonging to its superlevel set and \emph{all possible dynamics uncertainty}, verify whether there always exists a control input to maintain the system inside the superlevel set.
    \item Synthesis. Find a valid robust CBF (and verify its correctness) from a given function space.
\end{enumerate}

For systems with {polynomial} dynamics and simple \HYedit{polyhedral bounds on the control}, current works convert the verification problem into a convex sum-of-squares (SOS) program~\cite{clark22arxiv-cbf,dai2022arxiv-clfcbfsynveri} and the synthesis problem into a nonconvex bilinear SOS program~\cite{wang2018acc-permissive,zhao22arxiv-cbfsos}. Two drawbacks exist in these approaches. First, the bilinear SOS program is solved via alternation which does not provide global convergence guarantees. Second, they do not handle dynamics uncertainty and general (convex) control constraints. 

\textbf{Contributions.}
We focus on robust CBF verification and synthesis for control-affine polynomial systems with {bounded} state-dependent additive uncertainty and convex \HYedit{polynomial} control constraints.
We first unify and formulate robust CBF verification and synthesis as \emph{multilevel} polynomial optimization problems (POP, Section~\ref{sec:formulation})~\cite{bennett22mp-hierarchical}. Particularly, given a parametric \HYedit{polynomial} robust CBF candidate, the verification problem is casted as a three-level POP hierarchically optimizing the uncertainty, control, and state. The synthesis problem then becomes a four-level POP with an additional search for the best parameter (Section~\ref{sec:formulation:multilevel}). Despite the intractability of multilevel optimization, we show that under convex control constraints and bounded uncertainty, the inner two levels of optimization (over the uncertainty and the control) can be eliminated in closed form via the KKT optimality conditions. Consequently, the verification problem reduces to a single-level POP over the state and the synthesis problem reduces to a min-max POP over the state and the parameter (Section~\ref{sec:formulation:reduction}). We then employ multilevel semidefinite programming (SDP) relaxations to approximately solve the POPs with \emph{asymptotic global convergence guarantees} (Section~\ref{sec:sdprelax}). Specifically, for the verification problem, we relax it via Lasserre's hierarchy of moment relaxations~\cite{lasserre01siopt-global}. For the synthesis problem, we adopt the two-stage relaxation in~\cite{lasserre11jgo-minmaxpop}: in the first stage we use SOS programming to find increasingly tight (piece-wise) polynomial lower bounds to the unknown \emph{value function} of the verification POP, and in the second stage we maximize the polynomial lower bounds to search for valid parameters leading to valid robust CBFs. The reader can refer to Fig.~\ref{fig:vanderpol_clean} for a pictorial illustration of our approach, where the dotted blue line is the unknown value function and the solid lines are the polynomial lower bounds found via SOS programming. We provide an in-depth study of our algorithm on the controlled Van der Pol oscillator (Section~\ref{sec:experiments}).
% \textbf{Our contributions.} Our main contributions are three folds. (1) We provide the first systematic work for robust CBF verification and synthesis with arbitrary convex control limits. Unlike other SOSP approaches, our method guarantees asymptotic global convergence. (2) Even before convergence, our min-max POP solver can provide a sufficient condition for a valid robust CBF's existence, which renders our algorithm computationally efficient. (3) We evaluate our algorithms on a van der Pol oscillator with uncertain dynamics and show its potential power on CBF synthesis problems. 

\textbf{Limitations.} (i) Our approach is limited by the current computational bottleneck in solving large-scale SDPs and can only handle low-dimensional systems. (ii) \HYedit{Our approach does not handle the composition of multiple CBFs yet.}

%!TEX root = ../main.tex
\section{Related Work}
\label{sec:relatedwork}
There are three fundamental problems in CBF: deployment, verification, and synthesis. The first focuses on designing a safe controller (online) when a valid CBF is given, while the latter two try to find such a valid CBF (offline).

{\bf Deployment}.
Online safe control given a CBF can be formulated as a quadratic program~\cite{ames2014cdc-cbforigin} in the absence of model uncertainty, and a second-order cone program~\cite{long2022ral-robustsocp,dhiman2021tac-robustsocp} or semi-infinite program~\cite{wei2022acc-uncertainsynthesis} when uncertainty exists. These convex programs can typically be solved efficiently.

{\bf Verification \& synthesis: SOS methods}. 
Sum-of-squares (SOS) programming has shown increasing potential in CBF verification and synthesis because it can handle infinite constraints and preserve computational tractability (in theory, SOS program is convex and can be solved in polynomial time, but practically it gets intractable when the size grows).
Early works \cite{prajna2004hscc-bfsynthesis,prajna2006automatica-bfsynthesis} consider barrier function synthesis without control input. With control input, classical methods formulate the synthesis problem as a bilinear SOS program, assuming a nominal~\cite{ames2019ecc-cbftheapp} or parameterized~\cite{wang2022arxiv-safetysynver} controller (which is a sufficient but unnecessary condition for a valid CBF). 
Recent works~\cite{clark22arxiv-cbf,dai2022arxiv-clfcbfsynveri,zhao22arxiv-cbfsos} remove the assumption of an explicit controller. For instance, \cite{zhao22arxiv-cbfsos} considers box-like control limits and formulates a single nonlinear SOS program. These works, however, share two shortcomings: (i) they base their algorithms upon alternation and lack global convergence guarantees; (ii) they do not handle dynamics uncertainty and general convex control constraints. Our framework aims to resolve these drawbacks.

{\bf Verification \& synthesis: sampling methods}.
Another popular line of research constructs a robust CBF by sampling a finite number of states and using Lipschitz conditions to bound the discretization error. They use evolutionary algorithms~\cite{wei2022acc-uncertainsynthesis}, constrained PAC learning~\cite{robey2021ifac-rcbfhybrid}, or convex optimization~\cite{lindemann2021arxiv-rcbfsafeexpert} to search for a robust barrier function. These algorithms require Lipschitz conditions and they can suffer from the curse of dimensionality (number of samples grows exponentially \wrt state dimension).

{\bf Other methods}. 
Handcrafted CBFs can be designed for specific systems such as controllable linear kscedit{systems}~\cite{clark2021automatica-controllablelinear} and Euler-Lagrange kscedit{systems}~\cite{cortez2020acc-eluerlag}. 
\cite{tonkens2022iros-refining}~uses Hamilton-Jacobi-Bellman reachability to iteratively refine a CBF, 
and~\cite{chen21cdc-backup} leverages a backup control policy.
Deep learning-based methods are also emerging. Safe reinforcement learning algorithms learn \cite{ma2022l4dc-saferl} or adapt \cite{chen2021lcss-saferl,westenbroek2021ifac-saferl} safety certificates along with a control policy without rigorous validation. 
\kscedit{Neural CBFs learn CBF candidates with neural networks~\cite{liu2023corl-safe}. Some of them also adopt a post-hoc validation after learning, such as satisfiability modulo theory~\cite{zhao2021fac-provableneuralcbf} and Lipschitz methods~\cite{jin2020arxiv-provableneuralcbf}.}
% Neural CBFs adopt a post-hoc validation after learning, such as satisfiability modulo theory~\cite{zhao2021fac-provableneuralcbf} and Lipschitz methods~\cite{jin2020arxiv-provableneuralcbf}. 
%!TEX root = ../main.tex

\section{A Multilevel POP Formulation for Verification \& Synthesis of Robust CBFs}
\label{sec:formulation}

Consider a control-affine system with additive uncertainty
\bea \label{eq:system}
\dot{x} = f(x) + g(x)u + J(x) \epsilon
\eea
where $x \in \Real{n}$ is the state, $u \in \bbU \subseteq \Real{m}$ is the control, $f(x):\Real{n} \mapsto \Real{n}$, $g(x): \Real{n} \mapsto \Real{n \times m}$, $J(x): \Real{n} \mapsto \Real{n \times d}$ and $\epsilon \in \Real{d}$ models the unknown disturbance.
We make the following assumptions for system~\eqref{eq:system}.

\begin{assumption}[Polynomial dynamics] \label{assume:dynamics}
The entries of $f,g,J$ are polynomials in $x$.
\end{assumption}

\begin{assumption}[Convex polynomial control constraints]\label{assume:control}
    $\bbU$ is a \HYedit{compact} convex set defined by finite polynomial inequality constraints, \HYedit{\ie $\bbU = \{ u \in \Real{m} \mid c_{u,i}(u) \leq 0, i=1,\dots,l_u \}$ where $\{c_{u,i}\}_{i=1}^{l_u}$ are convex polynomials. Moreover, there exists a $u_0$ such that $c_{u,i} (u_0) < 0$ for all $i=1,\dots,l_u$.}
\end{assumption}
This assumption is quite general as it includes typical control sets such as polytopes~\cite{dai2022arxiv-clfcbfsynveri}, boxes~\cite{zhao22arxiv-cbfsos}, and ellipsoids.

\begin{assumption}[Bounded uncertainty]\label{assume:epsilon}
$\norm{\epsilon} \leq M_{\epsilon}$.
\end{assumption}

\subsection{Robust Control Barrier Function}
A robust control barrier function is defined as follows.

\begin{definition}[Robust CBF]\label{def:robustcbf}
Let $b(x): \Real{n} \mapsto \Real{}$ be a smooth function and $\calC \doteq \{ x \in \Real{n} \mid b(x) \geq 0 \}$ be its \HYedit{compact} superlevel set. Then, $b(x)$ is a robust CBF for system~\eqref{eq:system} if there exists a class-$K$ function $\alpha$\footnote{\kscedit{$\alpha$ is a class-$K$ function if it is strictly increasing and $\alpha(0) = 0$.}} such that for all $x \in \calC$
\bea\label{eq:rcbforg}
\max_{u \in \bbU} \min_{\norm{\epsilon} \leq M_\epsilon } \dot{b}(x) \geq - \alpha (b(x)).
\eea
\kscedit{In other words}, if the system~\eqref{eq:system} starts inside $\calC$, then there exists a sequence of control such that the system trajectory remains inside $\calC$, regardless of the value of $\epsilon$. Due to Nagumo‘s Theorem~\cite{blanchini99automatica-set}, \eqref{eq:rcbforg} holds for all $x \in \calC$ if and only if
\bea 
\max_{u \in \bbU} \min_{\norm{\epsilon} \leq M_\epsilon} \dot{b}(x) \geq 0, \quad \forall x \in \partial \calC,
\eea
\ie there exists $u$ to pull the state back to $\calC$ whenever the state lies on the boundary of $\calC$ ($\partial\calC \doteq \{x\in\Real{n} \mid b(x) =0 \}$).
\end{definition}

\subsection{A Hierarchical POP Formulation}
\label{sec:formulation:multilevel}
Given a parametric \kscedit{CBF candidate} $b(x,\theta)$ \HYedit{that is polynomial in $x$ and $\theta$, and assume the parameter $\theta$ belongs to a compact semialgebraic set $\Theta \subseteq \Real{k}$, \ie $\Theta$ is defined by finite polynomial (in-)equalities}. We consider verifying and synthesizing a robust CBF from $b(x,\theta)$. By Definition~\ref{def:robustcbf}, we can formulate the following optimization problems. 

\begin{problem}[Verification]\label{prob:cbfverification} Fix $\theta$, define 
\bea \label{eq:verifyvaluefun}
V(\theta) := \min_{x \in \partial \calC} \max_{u \in \bbU} \min_{\norm{\epsilon} \leq M_\epsilon} \dot{b}(x,\theta).
\eea
If $V(\theta) \geq 0$ for a given $\theta$, then $b(x,\theta)$ is a robust CBF. Otherwise, a global minimizer $x^\star$ of~\eqref{eq:verifyvaluefun} with $V(\theta) < 0$ acts as a witness that $b(x,\theta)$ is not a robust CBF.
\end{problem}

$V(\theta)$ is called the \emph{value function} of the verification POP.

\begin{problem}[Synthesis]\label{prob:cbfsynthesis} Let $V(\theta)$ be as in~\eqref{eq:verifyvaluefun}, define
\bea\label{eq:cbfsynthesis}
V^\star = \max_{\theta \in \Theta} V(\theta), \quad \theta^\star \in \argmax_{\theta \in \Theta}V(\theta).
\eea
If $V^\star \geq 0$, then $b(x,\theta^\star)$ is a robust CBF. Otherwise ($V^\star < 0$),  there does not exist a robust CBF from the family $b(x,\theta)$.
\end{problem}

\kscedit{Note that, for verification}, any $x \in \partial \calC$ (\kscedit{not necessarily $x^\star$}) with ``$\max_u \min_\epsilon \dot{b}(x,\theta) < 0$''  can refute $b(x,\theta)$ as a valid CBF; and for synthesis, any $\theta$ (\kscedit{not necessarily $\theta^\star$}) with $V(\theta) \geq 0$ leads to a valid CBF $b(x,\theta)$. In fact, our synthesis method can return a set of valid $\theta$ with $V(\theta) \geq 0$. However, we choose to state our definitions as in Problems~\ref{prob:cbfverification}-\ref{prob:cbfsynthesis} to make it easier to streamline our algorithm.

\subsection{Reduction to Single-level and Min-max POP}
\label{sec:formulation:reduction}
Problems~\eqref{eq:verifyvaluefun} and~\eqref{eq:cbfsynthesis} are instances of \emph{hierarchical optimization} problems~\cite{bennett22mp-hierarchical}, which are in general very difficult to analyze and solve. Nonetheless, thanks to Assumptions~\ref{assume:control} and~\ref{assume:epsilon}, we will show that~\eqref{eq:verifyvaluefun} can be reduced to a single-level POP, while~\eqref{eq:cbfsynthesis} can be reduced to a min-max POP.

We start by expanding $\dot{b}(x,\theta)$:
\bea \label{eq:dotbexpand}
\hspace{-2mm}
\dot{b}(x,\theta)\! =\! \underbrace{\frac{\partial b}{\partial x}(x,\theta) f(x)}_{:= L_f b(x,\theta)}\! +\! \underbrace{\frac{\partial b}{\partial x}(x,\theta) g(x)}_{:=L_g b(x,\theta)} u\! +\! \underbrace{\frac{\partial b}{\partial x}(x,\theta) J(x)}_{:=L_J b(x,\theta)} \epsilon.\!\!\!
\eea
With~\eqref{eq:dotbexpand}, we develop $V(\theta)$ in~\eqref{eq:verifyvaluefun}:
\bea 
& \displaystyle\min_{x \in \partial \calC} \max_{u \in \bbU} \min_{\norm{\epsilon} \leq M_\epsilon} L_f b(x,\theta) + L_g b(x,\theta)u + L_J b(x,\theta) \epsilon \nonumber\\
= &  \displaystyle \min_{x \in \partial \calC} \max_{u \in \bbU} \bracket{L_f b(x,\theta)\! +\! L_g b(x,\theta)u\! +\! \min_{\norm{\epsilon}\! \leq M_\epsilon} L_J b(x,\theta) \epsilon  } \label{eq:developVstepone}\\
 = & \displaystyle \min_{x \in \partial \calC} \bracket{ L_f b(x,\theta)\! +\! \underbrace{\max_{u \in \bbU} L_g b(x,\theta) u}_{:=V_u^\star}\! +\! \underbrace{\min_{\norm{\epsilon} \leq M_\epsilon} L_J b(x,\theta) \epsilon}_{:=V_\epsilon^\star} } \label{eq:developVsteptwo}
\eea
where~\eqref{eq:developVstepone} holds because ``$L_f b(x,\theta)$'' and ``$L_g b(x,\theta) u$'' are constants \wrt \kscedit{``$\min_{\norm{\epsilon} \leq M_\epsilon}$''}; and~\eqref{eq:developVsteptwo} holds because ``$L_f b(x,\theta)$'' and ``$\min_{\norm{\epsilon} \leq M_\epsilon} L_J b(x,\theta) \epsilon$'' are constants \wrt ``$\max_{u \in \bbU}$''. We next show that both $V_\epsilon^\star$ and $V_u^\star$ in~\eqref{eq:developVsteptwo} can be solved in closed form.

{\bf (1) $V_\epsilon^\star$:}
$\norm{\epsilon}\leq M_\epsilon$ defines a $d$-dimensional ball with radius $M_\epsilon$ and it is easy to verify that choosing
\bea 
\epsilon^\star = \begin{cases}
    - M_\epsilon \frac{\kscedit{L_J b(x,\theta)\tran}}{ \norm{L_J b(x,\theta)}} & \text{if } \norm{L_J b(x,\theta)} \neq 0 \\
    \text{arbitrary} & \text{otherwise}
\end{cases}
\eea
leads to
\bea \label{eq:Vepsstar}
V^\star_\epsilon = - M_\epsilon \norm{L_J b(x,\theta)}.
\eea 

{\bf (2) $V_u^\star$:} 
according to Assumption~\ref{assume:control}, 
$u^\star$ is optimal for ``$\max_{u \in \bbU} L_g b(x,\theta) u$'' if and only if there exists a dual variable $\zeta^\star \in \Real{l_u}$ such that $(u^\star, \zeta^\star)$ satisfies the following \kscedit{Karush–Kuhn–Tucker (KKT)} optimality conditions~\cite{boyd04book-convex}:
\begin{subequations}\label{eq:kktconds}
    \begin{eqnarray}
        \hspace{-4mm} \text{\grayout{primal feasibility:} }& c_{u,i}(u) \leq 0,i=1,\dots,l_u \label{eq:kktprimal}\\
        \hspace{-4mm} \text{\grayout{dual feasibility:} }& \zeta_i \geq 0, i=1,\dots,l_u \label{eq:kktdual}\\
        \hspace{-4mm} \text{\grayout{stationarity:} }& - L_g b(x,\theta)\! +\! \displaystyle \sum_{i=1}^{l_u} \zeta_i \frac{\partial c_{u,i}}{\partial u}(u)\! =\! 0 \label{eq:kktstationary} \\
        \hspace{-4mm} \text{\grayout{complementarity:} }& \zeta_i c_{u,i}(u) = 0,i=1,\dots,l_u. \label{eq:kktcomp}
    \end{eqnarray}
\end{subequations}
Observe that~\eqref{eq:kktconds} is a set of polynomial (in-)equalities.
Let $\bbK(x,\theta) \subseteq \Real{m} \times \Real{l_u}$ be the set of optimal $u$ and $\zeta$ defined by~\eqref{eq:kktconds}, we have that
\bea \label{eq:Vustar}
V_{u}^\star = L_g b(x,\theta) u^\star, \quad (u^\star,\zeta^\star) \in \bbK(x,\theta).
\eea 
\begin{remark} The optimal control $u^\star$ is generally a nonsmooth function of $x$ and $\theta$. Consider $c_{u,i} = u_i^2 - 1,i=1,\dots,m$, \ie $\bbU = [-1,1]^m$ is an $m$-D box, $u^\star$ is Bang-Bang in each of its $m$ dimensions, depending on the sign of $L_g b(x,\theta)$. Therefore, (i)  methods assuming $u^\star$ to be smooth (\eg a polynomial~\cite{jarvis03cdc-some}) are conservative; (ii) recent work~\cite{zhao22arxiv-cbfsos} synthesizes CBF by considering all $2^m$ combinations of Bang-Bang controllers. By explicitly introducing dual variables $\zeta$, our derivation uses KKT conditions to bridge nonsmooth control and smooth polynomial optimization.
\end{remark}

Plugging~\eqref{eq:Vepsstar} and~\eqref{eq:Vustar} into~\eqref{eq:developVsteptwo}, we have that 
\bea 
\hspace{-2mm}V(\theta) =\!\!\!\! \min_{\substack{x \in \Real{n}, z \in \Real{} \\ u^\star \in \Real{m}, \zeta^\star \in \Real{l_u}}} & \!\!\!\!\! L_f b(x,\theta) + L_g b(x,\theta) u^\star - M_\epsilon z \label{eq:verifypop}\\
\subject & z^2 = \norm{L_J b(x,\theta)}^2 \label{eq:zlift}\\
        & b(x,\theta) = 0 \label{eq:xcon}\\
        & (u^\star, \zeta^\star) \in \bbK(x,\theta)
\eea
where $x \in \partial\calC$ is explicitly written as~\eqref{eq:xcon}; \eqref{eq:zlift} enforces $z = \pm \norm{L_J b(x,\theta)}$ and the ``$\min$'' in the objective~\eqref{eq:verifypop} will push $z = \norm{L_J b(x,\theta)}$ (and hence~\eqref{eq:Vepsstar} is implicit).

{\bf Standard POP}. Denote $y = [x;z;u^\star;\zeta^\star] \in \Real{N}$ with $N=n+1+m+l_u$, 
% and recall $x \in \bbX$ can be described by polynomial (in-)equalities (Assumption~\ref{assume:state}), 
we can convert the verification problem~\eqref{eq:verifypop} to the following standard POP
\bea \label{eq:standardpop}
V(\theta) = \min_{y \in \Real{N}} \cbrace{ \varphi(y,\theta) \mymid \substack{ \displaystyle h_i(y,\theta) = 0, i=1,\dots,l_h \\[1mm] \displaystyle s_i(y,\theta) \geq 0,i=1,\dots,l_s } }.
\eea
Consequently, the original synthesis problem~\eqref{eq:cbfsynthesis} is equivalent to a min-max POP (where the ``$\max$'' is over $\theta \in \Theta$).

\HYedit{
\begin{remark}[Non-polynomial Dynamics]
Our framework is not restricted to polynomial dynamics. When the original dynamics~\eqref{eq:system} is non-polynomial, our framework still applies as long as the verification problem~\eqref{eq:verifypop} can be converted into a POP via a change of variables. For instance, assume the dynamics~\eqref{eq:system} and the CBF candidate $b(x,\theta)$ contain trigonometric terms in $\bar{x} \in x$, so long as the functions in problem~\eqref{eq:verifypop} are polynomials in $\sin(\bar{x})$ and $\cos(\bar{x})$, creating $\snew = \sin{\bar{x}},\cnew=\cos{\bar{x}}$ can turn~\eqref{eq:verifypop} into a POP with an extra polynomial constraint $\snew^2 + \cnew^2 = 1$. However, we stated Assumption~\ref{assume:dynamics} in the beginning to simplify our presentation. 
\end{remark}    
}

%!TEX root = ../main.tex
\section{Semidefinite Relaxation}
\label{sec:sdprelax}
We now apply semidefinite relaxations to solve the verification problem (Section~\ref{sec:sdpverify}) and the synthesis problem (Section~\ref{sec:sdpsynthesis}). \HYedit{Most of the results presented in this section are adapted from existing techniques proposed in~\cite{lasserre01siopt-global,lasserre11jgo-minmaxpop}. Our contribution here is to draw connections, for the first time, from global optimization of (min-max) POPs to verification and synthesis of robust CBFs. We hope these connections will inspire researchers working on similar problems.} 

Before diving into the details, we add $s_0 := 1$ into POP~\eqref{eq:standardpop} and assume the following conditions. 
\begin{assumption}[\kscedit{Feasible and Archimedean sets}]
    \label{assume:archimedeanness}
    For any $\theta \in \Theta$, (i) the feasible set of POP~\eqref{eq:standardpop} is non-empty; (ii) there exists $M_y > 0$ such that $M_y - \norm{y}^2 = \sum_{i=1}^{l_h} \mu_i h_i + \sum_{i=0}^{l_s} \sigma_i s_i$ holds for some polynomials $\{\mu_i \}_{i=1}^{l_h}$ and sum-of-squares (SOS) polynomials $\{ \sigma_i \}_{i=0}^{l_s}$ in $y$.
\end{assumption}

\HYedit{
   Assumption~\ref{assume:archimedeanness} is very general. First, the feasible set of~\eqref{eq:standardpop} is non-empty as long as $\partial \calC = \{x\in \Real{n} \mid b(x,\theta) = 0\}$ is non-empty, which can be satisfied via the design of $b(x,\theta)$. Second, if we know \emph{a priori} a bound $M_y$ on $y$, then adding a redundant constraint $M_y - \norm{y}^2 \geq 0$ to~\eqref{eq:standardpop} makes the Archimedean condition trivially satisfied. Because $\partial \calC$ and $\bbU$ are compact, it is easy to see that both $x$ and $u^\star$ are bounded. From~\eqref{eq:zlift} we observe $z$ is bounded because $z^2$ is a smooth function of $x$ (and $x$ belongs to a compact set). The next Proposition gives sufficient conditions for when the optimal dual variable $\zeta^\star$ is also bounded. 

\begin{proposition}[Bounded $\zeta^\star$]
    \label{prop:boundeddual}
If $\bbU$ is one of the following:
\begin{enumerate}[label=(\roman*)]
    \item Polytope: $c_{u,i}(u) = w_i\tran u + d_i \leq 0, i=1,\dots, l_u$ with no degenerate extreme points; 
    \item Box: $c_{u,i} (u) = u_i^2 - w_i^2, i=1,\dots,m$ with $w_i > 0$;
    \item Ellipsoid: a single $c_{u} = u\tran W u - 1$ with $W \succ 0$, 
\end{enumerate}
then there exists a constant $M_\zeta$ such that $\norm{\zeta^\star} \leq M_\zeta$.
\end{proposition}
\begin{proof}
    See Appendix~\ref{app:proof:prop:boundeddual}. \hfill \qedsymbol
\end{proof}
Section~\ref{sec:experiments} computes specific bounds for our test problem.}

\subsection{Verification: Lasserre's Hierarchy}
\label{sec:sdpverify}
We apply Lasserre's hierarchy of moment-SOS semidefinite relaxations~\cite{lasserre01siopt-global} to solve the POP~\eqref{eq:standardpop}. Due to space constraints, we only give a brief overview and refer the interested reader to~\cite{lasserre01siopt-global} or~\cite[Section 2.2]{yang22pami-certifiably} for details. 

Let $t_{\max} = \max\{\deg{\varphi}, \{\deg{h_i}\}_{i=1}^{l_h}, \{ \deg{s_i} \}_{i=0}^{l_s}  \}$ be the maximum degree of the POP~\eqref{eq:standardpop} with a fixed $\theta$, $\kappa \in \bbN$ be any integer such that $2\kappa \geq t_{\max}$, and $\kappa_0$ be the smallest such $\kappa$. Denote by $[y]_\kappa$ the vector of monomials in $y$ with degree up to $\kappa$, and by $Y_{\kappa} = [y]_\kappa [y]_\kappa\tran$ the moment matrix in $y$ of order $\kappa$. Further, let
\bea 
S_i = s_i(y,\theta) \cdot Y_{\kappa - \ceil{\deg{s_i}/2}},i=1,\dots,l_s
\eea
be the localizing matrix associated with $s_i$ (so that all the monomials in $S_i$ have degree at most $2\kappa$). With $Y = (Y_\kappa, S_1,\dots,S_{l_s})$, consider the semidefinite program (SDP):
\bea \label{eq:sdpverify}
\rho_{\kappa} = \min_{Y} \{ \inprod{C}{Y} \mid Y \succeq 0,\ \ \calA(Y) = e \}
\eea
where $C = (C_0,C_1,\dots,C_{l_s})$ is constant, has the same size as $Y$, and $\inprod{C}{Y} = \varphi(y,\theta)$;\footnote{$\inprod{C}{Y} := \trace{C_0 Y_0} + \dots + \trace{C_{l_s} S_{l_s}}$. We can have $C_1 = \dots = C_{l_s} = 0$ and $C_0$'s entries be the coefficients of $\kscedit{\varphi(y,\theta)}$.} $\calA(Y) = e$ collects all linear dependencies in $Y$ (\eg entries of $S_i$ are linear combinations of entries of $Y_\kappa$). It is clear that SDP~\eqref{eq:sdpverify} is a convex relaxation of the POP~\eqref{eq:standardpop} because every feasible point $y$ of the POP can generate a $Y$ that is feasible for the SDP via the moment and localizing matrices. The following theorem states that, as $\kappa \rightarrow \infty$, solving the SDP can recover the global optimizers of the POP.

\begin{theorem}[Lasserre's hierarchy~\cite{lasserre01siopt-global,henrion05-detecting}]
\label{thm:lasserre}
Let $\rho_\kappa^\star$ and $Y^\star = (Y_\kappa^\star, S_1^\star,\dots,S_{l_s}^\star)$ be the optimal value and one optimal solution of the SDP~\eqref{eq:sdpverify}, then
\begin{enumerate}[label=(\roman*)]
    \item $\rho_\kappa^\star \leq V(\theta)$ for any $\kappa$, and $\rho_\kappa^\star \rightarrow V(\theta)$ as $\kappa \rightarrow \infty$;
    \item if $Y_\kappa^\star$ satisfies the flatness condition, \ie $r=\rank{ Y_{\kappa' - \kappa_0}^\star } = \rank{ Y_{\kappa'}^\star }$ for some $\kappa_0 \leq \kappa' \leq \kappa$, then $\rho^\star_\kappa = V(\theta)$ and the relaxation is said to be tight or exact. Furthermore, $r$ global optimizers of the POP~\eqref{eq:standardpop} can be extracted from $Y_\kappa^\star$.
\end{enumerate}
\end{theorem}

Although the convergence in Theorem~\ref{thm:lasserre} is asymptotic, many practical problems~\cite{yang22mp-inexact} \kscedit{observe finite convergence}, \ie $\rho_\kappa^\star$ coincides with the global optimum of the original POP at a finite (and often small) relaxation order $\kappa$, and this can be proved with additional assumptions of the POP~\cite{nie14mp-optimality}. For the purpose of verification, $\rho^\star_\kappa \geq 0$ for some $\kappa$ is sufficient to certify the correctness of a robust CBF. 

% \red{In Section~\ref{sec:experiments}, we will show that, when $b(x,\theta)$ is not a valid CBF, solving the POP to global optimality helps us find an $x$ that refutes $b(x,\theta)$.} 

\subsection{Synthesis: Polynomial Approximation}
\label{sec:sdpsynthesis}
A naive approach to synthesize a robust CBF $b(x,\theta)$ is to randomly sample parameters $\theta \in \Theta$ until $V(\theta) \geq 0$. This approach can work well if $\Theta$ is a finite set. Another potential approach is to employ differentiable optimization, \ie computing $\partial V(\theta) / \partial \theta$ after solving the verification problem and performing gradient ascent to optimize $\theta$. Since the verification problem is a nonconvex POP, the convergence of this approach in solving the min-max synthesis problem is unclear (differentiable optimization typically works well when the inner problem is convex~\cite{bennett22mp-hierarchical}). In the following, we introduce \HYedit{the method proposed in~\cite{lasserre11jgo-minmaxpop} that is based on polynomial approximation}. The goal is to use a set of polynomials to lower bound $V(\theta)$, and use Lasserre's hierarchy to maximize the polynomial lower bounds. 
% As a result, this approach is a two-stage semidefinite relaxation.
% Let $\tldTheta \subset \Real{k}$ be a \emph{simple} set, \eg a box or a ball, that contains $\Theta$, and let $\tldTheta$ be defined by polynomial inequalities:
% \bea
% \tldTheta = \{ \theta \in \Real{k} \mid s_i(\theta) \geq 0,i=l_s + 1,\dots,\tldl_s \}.
% \eea 

Let $\psi$ be the uniform distribution supported on $\Theta$ so that 
\bea \label{eq:computemoments}
\gamma_{\beta} = \int_{\Theta} \theta^\beta d\psi(\theta) , \quad \beta \in \bbN^k
\eea
can be computed for any monomial $\theta^\beta = \theta_1^{\beta_1} \theta_2^{\beta_2}\cdots \theta_k^{\beta_k}$.\footnote{If $\gamma_\beta$ cannot be computed on $\Theta$, then one can augment $\Theta$ into a simple set $\tldTheta \supset \Theta$ (\eg a box or a ball) such that $\gamma_\beta$ is easy to compute on $\tldTheta$.}
% We make the following assumption on $\tldTheta$, and will show it holds for our test problem in Section~\ref{sec:experiments}.
% \begin{assumption}[Nonempty feasible set]
%     \label{assume:nonempty}
%     For any $\theta \in \tldTheta$, the feasible set of the POP~\eqref{eq:standardpop} is nonempty.
% \end{assumption}
Let $\nu \in \bbN$ such that $2\nu$ is no smaller than the maximum degree of the POP~\eqref{eq:standardpop}, and let $\nu_0$ be the smallest such $\nu$.\footnote{Note that in general $\nu_0 \neq \kappa_0$ introduced in Section~\ref{sec:sdpverify}. This is because $\theta$ is considered as constant when counting the maximum degree of the POP~\eqref{eq:standardpop} for verification, while considered the same as $y$ as an unknown variable when performing synthesis. For example, $\psi(y,\theta) = y^2 \theta^4$ has degree $2$ in $y$, but degree $6$ in $[y;\theta]$.} Denote by $\bbN^{k}_{2\nu}$ the set of $k$-dimensional integers summing up to $2\nu$.
Consider the following sum-of-squares (SOS) problem
\bea \label{eq:soslowerbound}
\hspace{-4mm} \max_{\lambda,\sigma,\mu} & \sum_{\beta \in \bbN^{k}_{2\nu}} \lambda_\beta \gamma_\beta \\
\hspace{-4mm} \subject & \substack{ \varphi(y,\theta) - \sum_{\beta \in \bbN^k_{2\nu}} \lambda_\beta \theta^\beta = \\ \sum_{i=0}^{\tldl_s} \sigma_i(y,\theta) s_i(y,\theta) + \sum_{i=1}^{l_h} \mu_i(y,\theta) h_i(y,\theta)  } \\
& \sigma_i \in \Sigma[y,\theta],\deg{\sigma_i s_i} \leq 2\nu, i=0,\dots,\tldl_s \\
& \mu_i \in \poly{y,\theta}, \deg{\mu_i h_i} \leq 2\nu, i=1,\dots,l_h
\eea
where $\poly{y,\theta}$ (resp. $\Sigma[y,\theta]$) is the set of real polynomials (resp. SOS polynomials) in $[y;\theta]$.
Let $\lambda^\star$ be \kscedit{its global optimizer (or one of its global optimizers)}. Denote
\bea \label{eq:polylowerbound}
V_{\nu}(\theta) := \sum_{\beta \in \bbN^k_{2\nu}} \lambda^\star_\beta \theta^\beta, \quad \nu \geq \nu_0
\eea 
as the polynomial in $\theta$ whose coefficients are $\lambda^\star$. The following theorem states that $V_{\nu}(\theta)$ is a polynomial lower bound for the unknown (and usually nonsmooth) $V(\theta)$. Moreover, $V_{\nu}(\theta)$ converges to $V(\theta)$ as $\nu$ increases.

\begin{theorem}[Global Convergence~\cite{lasserre11jgo-minmaxpop}]
    \label{thm:globalconvergence}
    Let Assumption~\ref{assume:archimedeanness} hold and $V_\nu (\theta)$ be as in \eqref{eq:polylowerbound}, we have 
    \begin{enumerate}[label=(\roman*)]
        \item $V_\nu(\theta) \leq V(\theta)$ for all $\theta \in \Theta$ and $\nu \geq \nu_0$;
        \item $\int_{\Theta} \abs{V(\theta) - V_\nu(\theta)} d\psi(\theta) \rightarrow 0$ as $\nu \rightarrow \infty$;
        \item let $\tldV_{\nu}(\theta) := \max\{ V_{\nu_0}(\theta),\dots,V_\nu(\theta) \}$, then $V_\nu(\theta) \leq \tldV_\nu(\theta) \leq V(\theta)$ for all $\theta \in \Theta$ and $\tldV_\nu(\theta) \rightarrow V(\theta)$, $\psi$-almost uniformly on $\Theta$ as $\nu \rightarrow \infty$.
    \end{enumerate}
    Furthermore, let 
    \bea \label{eq:maxlowerbound}
    V_\nu^\star := \max_{\theta \in \Theta} V_\nu(\theta), \quad \nu \geq \nu_0
    \eea 
    and $\theta^\star_\nu$ be a global optimizer. Denote
    \bea \label{eq:bestmaximizer}
    \hatV_{\nu}^\star := \max_{\nu_0 \leq l \leq \nu} V^\star_l = V_{\tau(\nu)}(\theta_{\tau(\nu)}^\star).
    \eea 
    for some $\tau(\nu) \in \{ \nu_0,\dots,\nu \}$. Then we have 
    \begin{enumerate}[label=(\roman*)]
        \setcounter{enumi}{3}
        \item $\hatV_{\nu}^\star \rightarrow V^\star$ in~\eqref{eq:cbfsynthesis} as $\nu \rightarrow \infty$
        \item if $V(\theta)$ is continuous on $\Theta$ then $V^\star = V(\theta^\star)$ for some $\theta^\star \in \Theta$; and any accumulation point $\bar{\theta}$ of the sequence $(\theta_{t(\nu)}^\star) \subset \Theta$ is a global optimizer of~\eqref{eq:cbfsynthesis}. In particular, if $\theta^\star$ is unique, then $\theta_{t(\nu)}^\star  \rightarrow \theta^\star$ as $\nu \rightarrow \infty$. 
    \end{enumerate}
\end{theorem}

A few remarks are in order about Theorem~\ref{thm:globalconvergence}. First, each $V_{\nu}(\theta)$ is a valid lower bound for $V(\theta)$ on the set $\Theta$ and as $\nu$ increases, $\int_{\Theta} \abs{V(\theta) - V_\nu(\theta)} d\psi(\theta)$ tends to zero. Second, assume one has computed a sequence of $V_{l}(\theta)$ for $l=\nu_0,\dots,\nu$, then the point-wise maximum $\tldV_{\nu}(\theta) = \max_{\nu_0 \leq l \leq \nu} V_l (\theta)$ is a tighter lower bound for $V(\theta)$ in that $\tldV_\nu (\theta)$ converges to $V(\theta)$ almost everywhere on $\Theta$. Third, if one can compute the global maximizers of the sequence of lower bounds $\{ V_{l}(\theta) \}_{l=\nu_0}^{\nu}$, then the best maximizer (\cf~\eqref{eq:bestmaximizer}) converges to the global optimizer of the min-max POP. To compute the global maximizer in~\eqref{eq:maxlowerbound}, we use Lasserre's hierarchy introduced in Section~\ref{sec:sdpverify}. Because~\eqref{eq:maxlowerbound} is a low-dimensional POP only in $\theta$, empirically it is easy to solve~\eqref{eq:maxlowerbound} to global optimality. Fourth, unlike Theorem~\ref{thm:lasserre} which states that it is possible to detect and certify global optimality via the flatness condition, Theorem~\ref{thm:globalconvergence} does not tell us how to numerically detect when $V^\star = \hatV^\star_{\nu}$ happens. This implies that our algorithm cannot claim the non-existence of a valid robust CBF in $\{b(x,\theta) \mid \theta \in \Theta \}$. Last but not the least, assuming $\{ \theta \mid V_\nu(\theta) \geq 0 \}$ is non-empty for some $\nu$, then the lower bound polynomial $V_\nu (\theta)$ enables us to choose a parameter $\theta$ that optimizes some performance metric $\Psi(\theta)$ via solving the POP ``$\min_{\theta \in \Theta}\{ \Psi(\theta) \mid V_\nu (\theta) \geq 0 \}$''. For example, as we will show in Section~\ref{sec:experiments}, $\Psi(\theta)$ can be chosen as (inversely) proportional to the volume of the set $\calC$. 
%!TEX root = ../main.tex

\section{Applications and Experiments}
\label{sec:experiments}

In this section, we apply our verification and synthesis method to the controlled Van der Pol oscillator. We first consider a Van der Pol oscillator without uncertainty (Section~\ref{sec:cleanvanderpol}) and show that our algorithm readily reproduces the circular CBF in~\cite{clark22arxiv-cbf}. With uncertainty (Section~\ref{sec:uncertainvanderpol}), we show how to synthesize elliptical robust CBFs with a strictly positive value function.

\subsection{A Gentle Start: Clean Van der Pol Oscillator}
\label{sec:cleanvanderpol}
Consider a clean controlled Van der Pol Oscillator~\cite{clark22arxiv-cbf}
\bea \label{eq:cleanvdpodynamics}
\dx = \bmat{c}
\dx_1 \\
\dx_2 \emat 
\!=\! \bmat{c}
x_2 \\
\frac{1}{2} (1-x_2^2) x_2  -x_1
\emat + \bmat{c}
0 \\
x_1 
\emat u.
\eea
Let 
\begin{equation}\label{eq:cleanvdpou}
        \bbU = \{u \in \Real{} \mid u^2 - u_\max^2 \leq 0 \}. 
\end{equation}

{\bf Circular CBF}. Consider a circular CBF candidate
\bea \label{eq:cleanvdpocbf}
b(x,\theta) = \theta - \norm{x}^2, \quad \theta \in \Theta := [0,\theta_\max].
\eea
\kscedit{Clearly, }Assumption~\ref{assume:archimedeanness} holds and Appendix~\ref{app:bound:y:cleanvdp} computes the bound for $y$ in POP~\eqref{eq:verifypop}.

{\bf Verification}. 
We choose $u_\max = 5, \theta_\max = 2$, and solve the $\kappa=4$ SDP relaxation of the verification POP~\eqref{eq:standardpop}.\footnote{We use the SDP relaxation implementation provided in \url{https://github.com/MIT-SPARK/CertifiablyRobustPerception}, and solve the SDP using MOSEK.} Fig.~\ref{fig:vanderpol_clean}(a) blue dotted line plots the values of the SDP relaxation at samples $\theta = 0,0.1,0.2,\dots,2.0$. In all cases, the SDP relaxation is tight and can recover the global optimizers of the verification POP~\eqref{eq:standardpop}. For example, when $\theta = 0.1$, $x^\star = [0;\pm 0.3162]$ attains the global minimum $V(\theta) = -0.09$; when $\theta = 1.1$, $x^\star = [\pm 1.0488;0]$ attains the global minimum $V(\theta) = 0$. From Fig.~\ref{fig:vanderpol_clean}(a) blue dotted line, we observe that $b(x,\theta)$ is a valid CBF for $\theta=0$ and $\theta \geq 1$, which agrees with the result from~\cite{clark22arxiv-cbf}. However, our verification algorithm can do better than that of~\cite{clark22arxiv-cbf}: when $b(x,\theta)$ is not valid for $\theta \in (0,1)$, our algorithm recovers the $x^\star$ that attains a negative $V(\theta)$, while the SOS-based method~\cite{clark22arxiv-cbf} becomes infeasible without producing a witness. 

%!TEX root = ../main.tex
\begin{figure}
    \begin{minipage}{\textwidth}
    \centering
    \begin{tabular}{cc}
        \begin{minipage}{0.49\textwidth}
            \centering
            \includegraphics[width=\columnwidth]{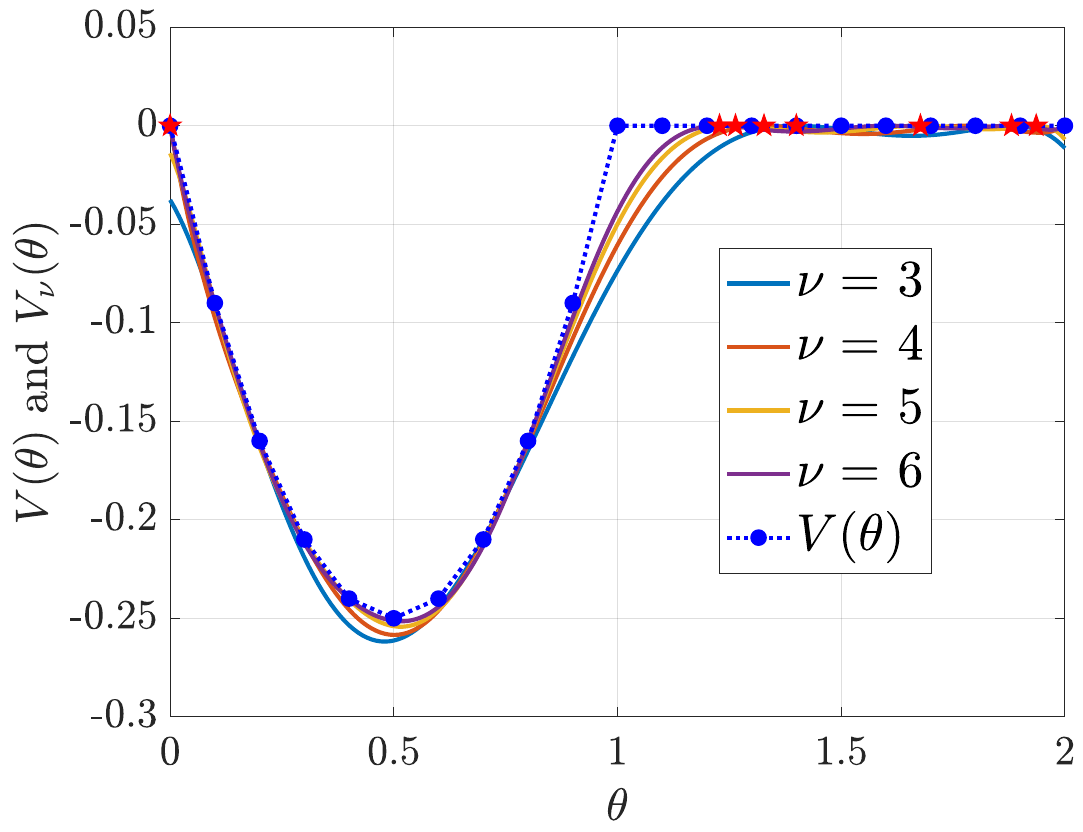}\\
            (a) $\Theta = [0,2]$
        \end{minipage}
        &
        \begin{minipage}{0.49\textwidth}
            \centering
            \includegraphics[width=\columnwidth]{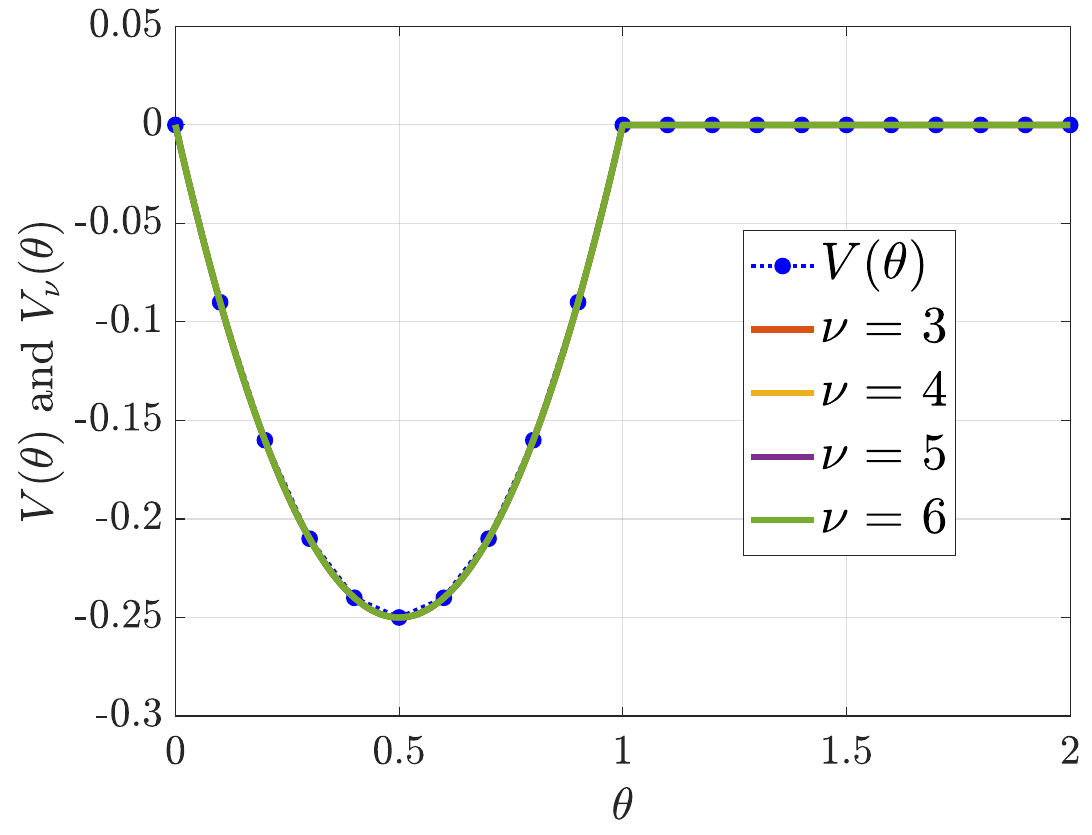}\\
            (b) $\Theta_1 = [0,1]$ and $\Theta_2 = [1,2]$
        \end{minipage}
    \end{tabular}
    \end{minipage}
    \caption{Verification and synthesis of a circular CBF $b(x,\theta) = \theta - \norm{x}^2$ for the clean Van der Pol oscillator~\eqref{eq:cleanvdpodynamics}. 
    \kscedit{
        Both figures plot SDP relaxation values $V(\theta)$ for verification and polynomial lower bounds $V_\nu(\theta)$'s for synthesis. When $V(\theta) \geq 0$, or $V_\nu(\theta) \geq 0$, it certifies $b(x,\theta)$ as a valid CBF. (a) shows the synthesis results when we calculate the lower bounds over $\Theta = [1, 2]$ directly. (b) shows that the lower bounds become much tighter when calculated in two subintervals $\Theta_1 = [0, 1]$ and $\Theta_2 = [1, 2]$ respectively.
    }
    \label{fig:vanderpol_clean}
    }
\end{figure}

{\bf Synthesis}.
We solve the SOS program~\eqref{eq:soslowerbound} at $\nu = 3,4,5,6$ to find polynomial lower bounds for $V(\theta)$. Appendix~\ref{sec:app:computemoments} describes how to compute the moments $\gamma_\beta$ in~\eqref{eq:computemoments}. Fig.~\ref{fig:vanderpol_clean}(a) solid lines plot the computed $V_\nu(\theta)$ for $\nu = 3,4,5,6$. The polynomial lower bounds approximate the unknown $V(\theta)$ very well (except near $\theta=1$ where $V(\theta)$ is nonsmooth). 

We then seek to maximize each $V_\nu(\theta)$ via Lasserre's hierarchy with $\kappa = \nu + 2$. At $\nu=3$, we obtain $\theta_\nu^\star = \{ 1.3997;1.8806 \}$ with $V_\nu^\star = \{1.2\times 10^{-8},3.2 \times 10^{-8} \}$. At $\nu=4$, we obtain $\theta_\nu^\star = \{ 1.3275 \}$ with $V_\nu^\star = \{-3.9\times 10^{-6} \}$. At $\nu=5$, we obtain $\theta_\nu^\star = \{ 1.2638, 1.6770, 1.9359 \}$ with $V_\nu^\star = \{-2.2\times 10^{-6}, -1.8\times 10^{-6}, 8.8 \times 10^{-6} \}$. At $\nu=6$, we obtain $\theta_\nu^\star = \{ 0,1.2280 \}$ with $V_\nu^\star = \{-7.8\times 10^{-6}, -7.4\times 10^{-8}\}$. These solutions are plotted as red stars in Fig.~\ref{fig:vanderpol_clean}(a).

{\bf Refined synthesis}. Since now we know the unknown $V(\theta)$ is nonsmooth at $\theta=1$, we can refine our synthesis by computing lower bound polynomials separately in $\Theta_1 = [0,1]$ and $\Theta_2 = [1,\theta_\max]$. Fig.~\ref{fig:vanderpol_clean}(b) shows the synthesized polynomial lower bounds $V_\nu (\theta)$ (we connect the curves in $\Theta_1$ and $\Theta_2$). Observe these lower bounds are perfectly tight!

Verification and synthesis of the one-dimensional $V(\theta)$ illustrates how our algorithm works. With these insights, we are ready to study the uncertain Van der Pol oscillator.

\subsection{Uncertain Van der Pol Oscillator}
\label{sec:uncertainvanderpol}
Consider system~\eqref{eq:cleanvdpodynamics} with uncertainty
\bea \label{eq:vdpodynamics}
\dx 
\!=\! \bmat{c}
x_2 \\
\frac{1}{2} (1-x_2^2) x_2  -x_1
\emat + \bmat{c}
0 \\
x_1 
\emat u + \bmat{c} 0 \\ 1 \emat \epsilon,
\eea 
with the same $\bbU$ as~\eqref{eq:cleanvdpou}. 

We can consider the same circular CBF candidate as in~\eqref{eq:cleanvdpocbf} and perform verification and synthesis. However, as we have seen in Section~\ref{sec:cleanvanderpol}, the maximum possible $\dot{b}(x,\theta)$ on the boundary $\partial \calC$ is zero, \ie $V^\star = 0$ for the synthesis problem~\eqref{eq:cbfsynthesis}. This is also true for the uncertain system.

\begin{proposition}\label{prop:wrongvdpocbf}
    $V(\theta) \leq 0, \forall \theta \geq 0$ for system~\eqref{eq:vdpodynamics} with the circular CBF candidate~\eqref{eq:cleanvdpocbf}.
\end{proposition}
\begin{proof}
    See Appendix~\ref{app:proofwrongvdpocbf}. \hfill \qedsymbol
\end{proof}

Can we find a robust CBF with strictly positive $V(\theta)$?

{\bf Elliptical robust CBF}.
Let us consider the following elliptical robust CBF candidate
\bea \label{eq:ellipsoidalcbf}
b(x,\theta)\! =\! 1\! -\! x\tran \underbrace{\bmat{cc}
\theta_1 & \theta_3 \\
\theta_3 & \theta_2 
\emat}_{:=A} x\!
\eea 
with parameter space
\bea \label{eq:ellipsoidalparam}
\Theta := \cbrace{\theta \in \Real{3} \mid \theta_1, \theta_2  \in [\thetalb,\thetaub]^2, \theta_3^2 \leq \xi^2 \theta_1 \theta_2},
\eea 
where $0 < \thetalb < \thetaub$ and $ 0 < \xi <1$. Clearly, $A \succ 0$ for all $(\theta_1,\theta_2,\theta_3) \in \Theta$ and $b(x,\theta)$ defines the boundary of an ellipse. Moreover, Assumption~\ref{assume:archimedeanness} holds and Appendix~\ref{app:bound:y:uncertainvdp} computes the bound for $y$ in POP~\eqref{eq:verifypop}. We will obtain $V(\theta) > 0 $ with this CBF candidate.

{\bf Synthesis}. Unlike Section~\ref{sec:cleanvanderpol} where it is easy to generate a few $\theta$ to evenly cover the parameter space $\Theta = [0,\theta_\max]$, the parameter space~\eqref{eq:ellipsoidalparam} would need a lot of samples to be covered. Therefore, we directly synthesize polynomial lower bounds $V_\nu(\theta)$ with $u_\max=5, M_\epsilon=0.1,\thetalb=0.25,\thetaub=0.75,\xi = 0.6$. (i) At $\nu = 3$, Fig.~\ref{fig:vanderpol_uncertain}(a) densely scatters samples of $\theta$ with $V_{\nu}(\theta)\geq 0$. The colormap shows that many samples (\eg the yellow points) have $V_{\nu}(\theta)$ being strictly positive. Particularly, the global maximum of $V_\nu(\theta)$ over $\Theta$ is attained at $\theta^\star_\nu = [0.2500;0.2757;0.1323]$ with value $V^\star_\nu = 0.2376$. (ii) Fig.~\ref{fig:vanderpol_uncertain}(b) shows the same scatter plot for $\nu = 4$. We see the set of elliptical robust CBFs that can be verified becomes much larger. The global maximum of $V_\nu(\theta)$ is attained at $\theta^\star_\nu = [0.3278;0.2500;0.1718]$ with value $V^\star_\nu = 1.0606$.

{\bf Smallest and largest ellipse}. What if we want to find the robust CBF that leads to the smallest ellipse? We can solve
\bea  
\max_{\theta \in \Theta} \{ \det{A} = \theta_1 \theta_2 - \theta_3^2  \mid V_\nu(\theta) \geq 0\}.
\eea 
Optimizing this for $V_3(\theta)$ via Lasserre's hierarchy at $\kappa=4$, we recover the best parameter $\theta_1=0.4714,\theta_2=0.5236,\theta_3 = 0.0893$. Similarly, $\theta_1=0.25,\theta_2=0.25,\theta_3 = 0.15$ minimizes $\det{A}$ and gives the largest ellipse.

%!TEX root = ../main.tex
\begin{figure}
    \begin{minipage}{\textwidth}
    \centering
    \begin{tabular}{cc}
        \begin{minipage}{0.49\textwidth}
            \centering
            \includegraphics[width=\columnwidth]{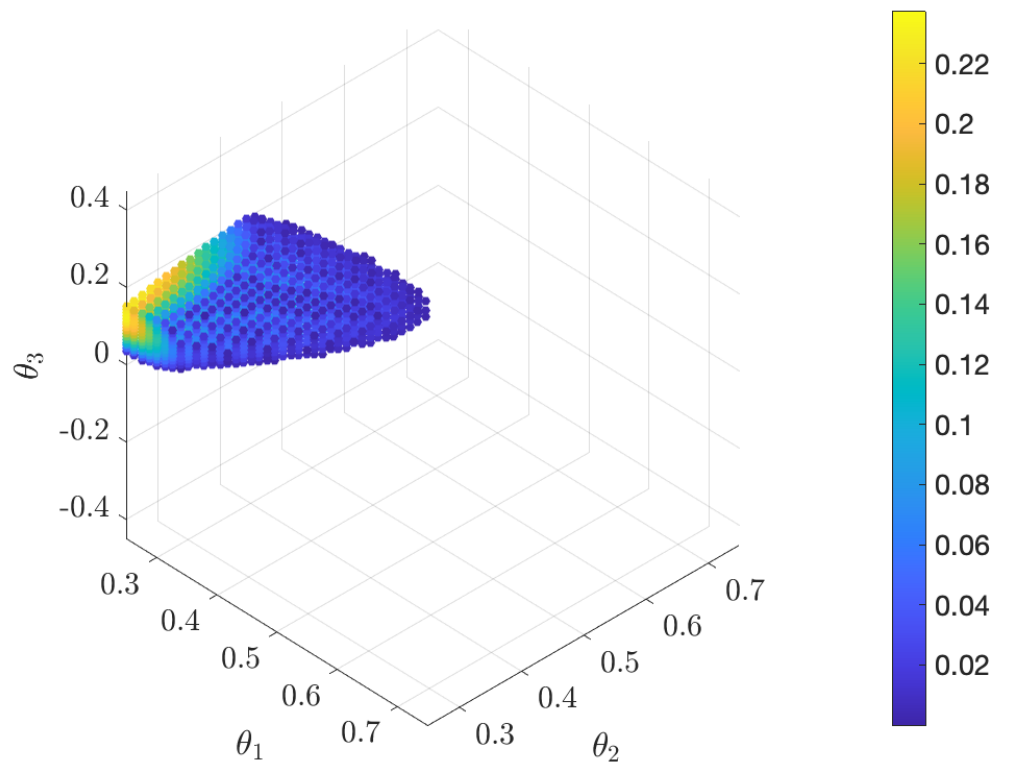}\\
            (a) $\nu = 3$
        \end{minipage}
        &
        \begin{minipage}{0.49\textwidth}
            \centering
            \includegraphics[width=\columnwidth]{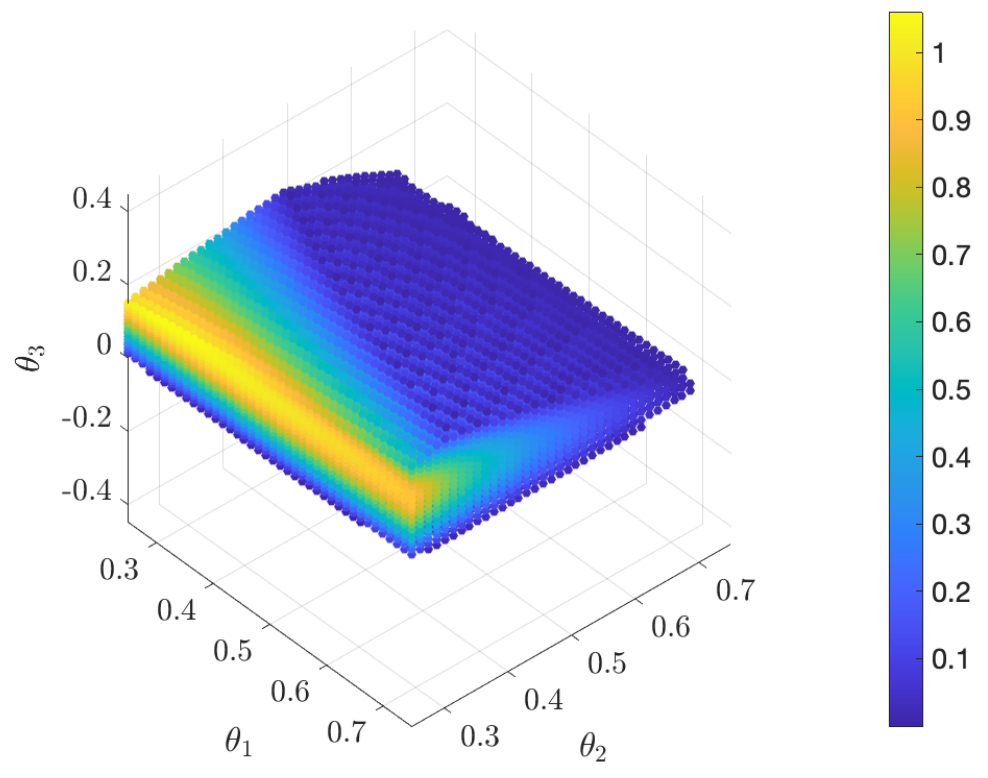}\\
            (b) $\nu = 4$
        \end{minipage}
    \end{tabular}
    \end{minipage}
    \caption{Synthesis of a robust ellipsoidal CBF $b(x,\theta) = 1 - x\tran A x$ for the uncertain Van der Pol oscillator~\eqref{eq:vdpodynamics}. 
    \kscedit{
        Both figures demonstrate a 3-D scatter plot of $V_\nu(\theta)$ against the 3-D parameter $\theta$. Each sampled $\theta$ with positive $V_\nu(\theta)$ is guaranteed to result in a valid robust CBF. For better visualization, we only scatter $\theta$ with $V_\nu(\theta) > 0$. (a) shows that a small fraction of $\theta$'s are certified to be valid with $\nu = 3$. (b) shows that the set of certifiable valid $\theta$'s becomes larger when we increase $\nu$ to $4$ (\ie using a tighter polynomial approximation).
    }
    \label{fig:vanderpol_uncertain}
    }
\end{figure}

%!TEX root = ../main.tex
\section{Conclusions}
We formulated robust CBF verification and synthesis as multilevel POPs. We used KKT conditions to eliminate inner problems in closed form, leading to a single-level POP formulation for verification and a min-max POP formulation for synthesis. We designed multilevel semidefinite relaxations to approximately solve both POPs with asymptotic global convergence guarantees. We validated our framework on a controlled Van der Pol oscillator. Future work will investigate extension of our method to verify multiple CBFs.

\section*{Acknowledgment}
The authors thank Jie Wang for providing an implementation that converts SOS programs into standard SDPs, and the anonymous reviewers for valuable comments and suggestions that help improve the quality of the paper.

%%%%%%%%%%%%%%%%%%%%%%%%%%%%%%%%%%%%%%%%%%%%%%%%%%%%%%%%%%%%%%%%%%%%%%%%%%%%%%%%

%!TEX root = ../main.tex
\appendix
\HYedit{
\section{Proof of Proposition~\ref{prop:boundeddual}}
\label{app:proof:prop:boundeddual}
\begin{proof}
(i) Consider $\bbU$ a polytope with no degenerate extreme points. Observe that ``$\max_{u \in \bbU} L_g b(x,\theta) u$'' is a linear program in $u$ with fixed $(x,\theta)$. Because $\bbU$ is non-degenerate, there are $m' \leq m$ active linearly independent constraints at an optimal $u^\star$, and hence $m'$ nonzero $\zeta_i^\star$'s~\cite{bertsimas97book-lp}. Without loss of generality, let $i=1,\dots,m'$ be such active constraints. The KKT stationarity condition~\eqref{eq:kktstationary} reads: 
\bea
\underbrace{\bmat{ccc}
w_1 & \cdots & w_{m'}
\emat}_{:=W \in \Real{m \times m'}} \zeta^\star = \kscedit{L_g b(x,\theta)\tran},
\eea 
which implies $(\zeta^\star)\tran (W\tran W) (\zeta^\star) = \norm{L_g b(x,\theta)}^2$. As $W\tran W \succ 0$ ($\{w_i\}_{i=1}^{m'}$ are linearly independent), we have 
\bea 
\norm{\zeta^\star}^2 \leq \frac{\norm{L_g b(x,\theta)}^2}{\lambda_\min (W\tran W)}.
\eea 
Finally, $\norm{L_g b(x,\theta)}^2$ is bounded because it is smooth over the compact set $\partial \calC \times \Theta$. Thus, $\norm{\zeta^\star}^2$ is bounded.

(ii) Consider $\bbU$ a box with each dimension between $[-w_i,w_i]$. The KKT complementarity condition~\eqref{eq:kktcomp} says either $\zeta_i^\star = 0$ or $\zeta_i^\star \neq 0$ but $u^\star_i = \pm w_i$. In the second case, using the KKT stationarity condition~\eqref{eq:kktstationary}, we have
\bea 
\zeta_i^\star = \frac{[L_g b(x,\theta)]_i}{2u^\star_i} \Rightarrow (\zeta_i^\star)^2 = \frac{[L_g b(x,\theta)]_i^2}{4w_i^2}, i=1,\dots,m
\eea 
where $[L_g b(x,\theta)]_i$ denotes the $i$-th entry of $L_g b(x,\theta)$. Due to the boundedness of $[L_g b(x,\theta)]_i$, $\norm{\zeta^\star_i}^2$ is bounded.

(iii) Consider $\bbU$ an ellipsoid with a single constraint $u\tran W u \leq 1$. The KKT complementarity condition states either $\zeta^\star = 0$ or $\zeta^\star \neq 0$ but $(u^\star)\tran W (u^\star) = 1$. In the second case, the KKT stationarity condition~\eqref{eq:kktstationary} reads
\bea 
(\zeta^\star)^2 = \frac{\norm{L_g b(x,\theta)}^2 }{4 \norm{W u^\star}^2} \leq \frac{\norm{L_g b(x,\theta)}^2}{4 \lambda_\min (W)}.
\eea
The boundedness of $(\zeta^\star)^2$ follows from the boundedness of $\norm{L_g b(x,\theta)}^2$ over $ \partial \calC \times \Theta$.
\hfill \qedsymbol
\end{proof}}

\section{Bounded $y$ for POP~\eqref{eq:standardpop} with system~\eqref{eq:cleanvdpodynamics} and CBF~\eqref{eq:cleanvdpocbf}}
\label{app:bound:y:cleanvdp}
Clearly $\norm{x}^2 \leq \theta, u^2 \leq u_\max^2$ are bounded. 
It remains to show that $\zeta^\star$ is bounded. The KKT complementarity condition~\eqref{eq:kktcomp} states that either (i) $\zeta^\star = 0$, or (ii) when $\zeta^\star \neq 0$, $c_{u,i}(u) = 0$, which means $u^\star= \pm u_\max$. In the second case, the KKT stationarity condition~\eqref{eq:kktstationary} leads to
\bea 
2\zeta^\star u^\star = L_g b(x,\theta) = -2 x_1 x_2.
\eea 
Solving for $\zeta^\star$ results in
\bea
\zeta^\star = \frac{-x_1 x_2}{u^\star} \Rightarrow (\zeta^\star)^2 = \frac{x_1^2 x_2^2}{u_\max^2} \leq \frac{(x_1^2 + x_2^2)^2}{4 u_\max^2} = \frac{\theta^2}{4u_\max^2}.
\eea

%%%%%%%%%%%%%%%%%%%%%%%%%%%%%%%%%%%%%%%%%%%%%%%%%%%%%%%%%%%
%%%%%%%%%%%%%%%%%%%%%%%%%%%%%%%%%%%%%%%%%%%%%%%%%%%%%%%%%%%

\kscedit{
\section{Computing Moments}
\label{sec:app:computemoments}
{\bf Clean Van der Pol with circular CBF}. Consider the parameter space $\Theta = [\theta_\min, \theta_\max]$ and let $\Delta_{\theta} = \theta_\max - \theta_\min$ be its length. We compute the moments in \eqref{eq:computemoments} as
\bea 
\gamma_\beta\! =\! \int_{\theta_\min}^{\theta^\max}\!\! \theta^{\beta} d \psi(\theta) \!=\! \frac{1}{\Delta_\theta}  \int_{\theta_\min}^{\theta^\max} \!\! \theta^{\beta} d \theta \!=\! \frac{(\theta_\max^{\beta+1} - \theta_\min^{\beta+1})}{(\beta + 1)\Delta_\theta}.
\eea 

{\bf Uncertain Van der Pol with elliptical robust CBF}.
Consider $\Theta$ as in~\eqref{eq:ellipsoidalparam}, we compute \eqref{eq:computemoments}:
\bea 
& \gamma_\beta = \displaystyle \int_{\Theta} \theta_1^{\beta_1} \theta_2^{\beta_2} \theta_3^{\beta_3} d \psi(\theta)\nonumber \\
= & \displaystyle \int_{\theta_1} \int_{\theta_2} \int_{\theta_3}  \theta_1^{\beta_1} \theta_2^{\beta_2} \theta_3^{\beta_3} \parentheses{\frac{1}{\volume{\Theta}} d\theta_1 d\theta_2 d\theta_3} \nonumber\\
=& \displaystyle \frac{1}{\volume{\Theta}} \int_{\theta_1} \int_{\theta_2} \theta_1^{\beta_1} \theta_2^{\beta_2} d\theta_1 d\theta_2 \parentheses{\int_{ -\xi \sqrt{\theta_1 \theta_2}}^{\xi \sqrt{\theta_1 \theta_2}}  \theta_3^{\beta_3}  d\theta_3} \nonumber\\
= &\!\!\!\!\! \displaystyle \frac{\xi^{\beta_3+1} (1 - (-1)^{\beta_3+1}) }{(\beta_3+1)\volume{\Theta}} \int_{\theta_1} \int_{\theta_2} \theta_1^{\beta_1 + \frac{\beta_3+1}{2}} \theta_2^{\beta_2 + \frac{\beta_3+1}{2}} d\theta_1 d\theta_2 \nonumber\\
=&  \displaystyle \frac{\xi^{\beta_3+1} (1 - (-1)^{\beta_3+1}) }{(\beta_3+1)\volume{\Theta}} \frac{
    \thetaub^{\tldbeta_1} - \thetalb^{\tldbeta_1}
}{\tldbeta_1}
\frac{
    \thetaub^{\tldbeta_2} - \thetalb^{\tldbeta_2}
}{\tldbeta_2}
\eea 
where $\volume{\Theta}$ is the volume of $\Theta$ (a constant that we do not need to compute) and $\tldbeta_1 = \beta_1 + 1 + \frac{\beta_3+1}{2}$, $\tldbeta_2 = \beta_2 + 1 + \frac{\beta_3+1}{2}$.
}

%%%%%%%%%%%%%%%%%%%%%%%%%%%%%%%%%%%%%%%%%%%%%%%%%%%%%%%%%%%
%%%%%%%%%%%%%%%%%%%%%%%%%%%%%%%%%%%%%%%%%%%%%%%%%%%%%%%%%%%

\section{Proof of Proposition~\ref{prop:wrongvdpocbf}}
\label{app:proofwrongvdpocbf}
\begin{proof}
    We have $L_f b(x,\theta) = -x_2^2 (1-x_2^2)$, $L_g b(x,\theta) = -2x_1 x_2$, $L_J b(x,\theta) = -2 x_2$. 
    As a result:
    \begin{subequations}
        \begin{eqnarray}
            V_u^\star =& \displaystyle \max_{u^2 \leq u_\max^2} (-2x_1 x_2) u = 2 u_\max \abs{x_1 x_2}  \\
            V_\epsilon^\star =& \hspace{-4mm}\displaystyle \min_{\norm{\epsilon} \leq M_\epsilon} -2x_2 \epsilon\! =\! -2 M_\epsilon \abs{x_2},
        \end{eqnarray}
    \end{subequations}
    and 
    \bea 
    \hspace{-2mm}V(\theta) =  \displaystyle \min_{\norm{x}^2 = \theta} -x_2^2(1-x_2^2) + 2 u_\max \abs{x_1 x_2} - 2 M_\epsilon \abs{x_2}.
    \eea
    Choosing $x_2 = 0,x_1 = \pm \sqrt{\theta}$, we obtain $V(\theta) \leq 0$. \hfill \qedsymbol
    % If $\theta \leq 1$, then with $x_1 = 0$ and $x_2 = \sqrt{\theta}$, we have $V(\theta) \leq - \theta (1-\theta) - 2 M_\epsilon \sqrt{\theta} < 0$. If $\theta > 1$, then 
\end{proof}

\section{Bounded $y$ for POP~\eqref{eq:standardpop} with system~\eqref{eq:vdpodynamics} and CBF~\eqref{eq:ellipsoidalcbf}}
\label{app:bound:y:uncertainvdp}
Clearly, $u^2 \leq u_\max^2$ is bounded.
$x$ is bounded because
\bea 
\max_{x\tran A x = 1} x\tran x \overset{v := A^{1/2}x}{=} \max_{v\tran v = 1} v\tran A\inv v = \frac{1}{\lambda_\min (A)},
\eea
where $\lambda_\min$ and $\lambda_\max$ indicates the minimum and maximum eigenvalue of $A$. Note that 
\bea 
\lambda_\min (\lambda_\max + \lambda_\min) \geq \lambda_\min \lambda_{\max} = \theta_1\theta_2 - \theta_3^2 \Longrightarrow \\
\lambda_\min \geq \frac{\theta_1 \theta_2 - \theta_3^2}{\lambda_\min + \lambda_\max} = \frac{\theta_1 \theta_2 - \theta_3^2}{\theta_1 + \theta_2}.
\eea 
Therefore, $x$ is bounded because
\bea \label{eq:uncertainvdpoxbound}
\norm{x}^2 \leq  \frac{1}{\lambda_\min (A)} \leq \frac{\theta_1 + \theta_2}{\theta_1 \theta_2 - \theta_3^2}.
\eea 
Now consider $z^2 = \norm{L_J b(x,\theta)}^2$:
\bea 
\norm{L_J b(x,\theta)}^2 = \norm{-2x\tran A J(x)}^2 = 4x\tran A J(x)J(x)\tran A x.
\eea 
Note that $J(x)J(x)\tran \preceq \eye$,
which means
\bea  \label{eq:uncertainvdpozbound}
\norm{L_J b(x,\theta)}^2 \leq 4x\tran A^2 x \leq 4\lambda_\max(A) \leq 4(\theta_1 + \theta_2).
\eea
It remains to show $\zeta^\star$ is bounded. Similar to \kscedit{Appendix}~\ref{app:bound:y:cleanvdp}, the KKT conditions tell us either $\zeta^\star = 0$ or 
\bea  
\zeta^\star = \frac{L_g b(x,\theta)}{2u^\star}, \quad (u^\star)^2 = u_\max^2.
\eea
Writing $(L_g b(x,\theta))^2$ as
\bea  
(L_g b(x,\theta))^2 = \norm{-2x\tran A g}^2 = 4x\tran A g g\tran A x
\eea 
with $g g\tran \preceq (x_1^2 + x_2^2) \eye$,
we obtain
\bea  
& \displaystyle (L_g b(x,\theta))^2 \leq 4 \norm{x}^2 x\tran A^2 x \leq 4\frac{(\theta_1 + \theta_2)^2}{\theta_1 \theta_2 - \theta_3^2}, \\
& \Rightarrow \displaystyle (\zeta^\star)^2 = \frac{(L_g b(x,\theta))^2}{4u_\max^2} \leq \frac{(\theta_1 + \theta_2)^2}{u_\max^2(\theta_1 \theta_2 - \theta_3^2)}.\label{eq:uncertainvdpozetabound}
\eea

\bibliographystyle{unsrt}
\bibliography{refs.bib}

\begin{thebibliography}{10}

\bibitem{ames2014cdc-cbforigin}
Aaron~D Ames, Jessy~W Grizzle, and Paulo Tabuada.
\newblock Control barrier function based quadratic programs with application to
  adaptive cruise control.
\newblock In {\em 53rd IEEE Conference on Decision and Control}, pages
  6271--6278. IEEE, 2014.

\bibitem{taylor2020acc-robustqp}
Andrew~J Taylor and Aaron~D Ames.
\newblock Adaptive safety with control barrier functions.
\newblock In {\em 2020 American Control Conference (ACC)}, pages 1399--1405.
  IEEE, 2020.

\bibitem{buch21csl-robust}
Jyot Buch, Shih-Chi Liao, and Peter Seiler.
\newblock Robust control barrier functions with sector-bounded uncertainties.
\newblock {\em IEEE Control Systems Letters}, 6:1994--1999, 2021.

\bibitem{clark22arxiv-cbf}
Andrew Clark.
\newblock A semi-algebraic framework for verification and synthesis of control
  barrier functions.
\newblock {\em arXiv preprint arXiv:2209.00081}, 2022.

\bibitem{dai2022arxiv-clfcbfsynveri}
Hongkai Dai and Frank Permenter.
\newblock Convex synthesis and verification of control-lyapunov and barrier
  functions with input constraints.
\newblock {\em arXiv preprint arXiv:2210.00629}, 2022.

\bibitem{wang2018acc-permissive}
Li~Wang, Dongkun Han, and Magnus Egerstedt.
\newblock Permissive barrier certificates for safe stabilization using
  sum-of-squares.
\newblock In {\em 2018 Annual American Control Conference (ACC)}, pages
  585--590. IEEE, 2018.

\bibitem{zhao22arxiv-cbfsos}
Weiye Zhao, Tairan He, Tianhao Wei, Simin Liu, and Changliu Liu.
\newblock Safety index synthesis via sum-of-squares programming.
\newblock {\em arXiv preprint arXiv:2209.09134}, 2022.

\bibitem{bennett22mp-hierarchical}
Kristin~P Bennett, Michael~C Ferris, Jong-Shi Pang, Mikhail~V Solodov, and
  Stephen~J Wright.
\newblock Special issue: Hierarchical optimization.
\newblock {\em Mathematical Programming}, pages 1--3, 2022.

\bibitem{lasserre01siopt-global}
Jean~B Lasserre.
\newblock Global optimization with polynomials and the problem of moments.
\newblock {\em SIAM J. Optim.}, 11(3):796--817, 2001.

\bibitem{lasserre11jgo-minmaxpop}
Jean~B Lasserre.
\newblock Min-max and robust polynomial optimization.
\newblock {\em Journal of Global Optimization}, 51(1):1--10, 2011.

\bibitem{long2022ral-robustsocp}
Kehan Long, Vikas Dhiman, Melvin Leok, Jorge Cort{\'e}s, and Nikolay Atanasov.
\newblock Safe control synthesis with uncertain dynamics and constraints.
\newblock {\em IEEE Robotics and Automation Letters}, 7(3):7295--7302, 2022.

\bibitem{dhiman2021tac-robustsocp}
Vikas Dhiman, Mohammad~Javad Khojasteh, Massimo Franceschetti, and Nikolay
  Atanasov.
\newblock Control barriers in bayesian learning of system dynamics.
\newblock {\em IEEE Transactions on Automatic Control}, 2021.

\bibitem{wei2022acc-uncertainsynthesis}
Tianhao Wei, Shucheng Kang, Weiye Zhao, and Changliu Liu.
\newblock Persistently feasible robust safe control by safety index synthesis
  and convex semi-infinite programming.
\newblock {\em IEEE Control Systems Letters}, 2022.

\bibitem{prajna2004hscc-bfsynthesis}
Stephen Prajna and Ali Jadbabaie.
\newblock Safety verification of hybrid systems using barrier certificates.
\newblock In {\em HSCC}, volume 2993, pages 477--492. Springer, 2004.

\bibitem{prajna2006automatica-bfsynthesis}
Stephen Prajna.
\newblock Barrier certificates for nonlinear model validation.
\newblock {\em Automatica}, 42(1):117--126, 2006.

\bibitem{ames2019ecc-cbftheapp}
Aaron~D Ames, Samuel Coogan, Magnus Egerstedt, Gennaro Notomista, Koushil
  Sreenath, and Paulo Tabuada.
\newblock Control barrier functions: Theory and applications.
\newblock In {\em 2019 18th European control conference (ECC)}, pages
  3420--3431. IEEE, 2019.

\bibitem{wang2022arxiv-safetysynver}
Han Wang, Kostas Margellos, and Antonis Papachristodoulou.
\newblock Safety verification and controller synthesis for systems with input
  constraints.
\newblock {\em arXiv preprint arXiv:2204.09386}, 2022.

\bibitem{robey2021ifac-rcbfhybrid}
Alexander Robey, Lars Lindemann, Stephen Tu, and Nikolai Matni.
\newblock Learning robust hybrid control barrier functions for uncertain
  systems.
\newblock {\em IFAC-PapersOnLine}, 54(5):1--6, 2021.

\bibitem{lindemann2021arxiv-rcbfsafeexpert}
Lars Lindemann, Alexander Robey, Lejun Jiang, Stephen Tu, and Nikolai Matni.
\newblock Learning robust output control barrier functions from safe expert
  demonstrations.
\newblock {\em arXiv preprint arXiv:2111.09971}, 2021.

\bibitem{clark2021automatica-controllablelinear}
Andrew Clark.
\newblock Control barrier functions for stochastic systems.
\newblock {\em Automatica}, 130:109688, 2021.

\bibitem{cortez2020acc-eluerlag}
Wenceslao~Shaw Cortez and Dimos~V Dimarogonas.
\newblock Correct-by-design control barrier functions for euler-lagrange
  systems with input constraints.
\newblock In {\em 2020 American Control Conference (ACC)}, pages 950--955.
  IEEE, 2020.

\bibitem{tonkens2022iros-refining}
Sander Tonkens and Sylvia Herbert.
\newblock Refining control barrier functions through hamilton-jacobi
  reachability.
\newblock In {\em 2022 IEEE/RSJ International Conference on Intelligent Robots
  and Systems (IROS)}, pages 13355--13362. IEEE, 2022.

\bibitem{chen21cdc-backup}
Yuxiao Chen, Mrdjan Jankovic, Mario Santillo, and Aaron~D Ames.
\newblock Backup control barrier functions: Formulation and comparative study.
\newblock In {\em IEEE Conf. on Decision and Control (CDC)}, pages 6835--6841.
  IEEE, 2021.

\bibitem{ma2022l4dc-saferl}
Haitong Ma, Changliu Liu, Shengbo~Eben Li, Sifa Zheng, and Jianyu Chen.
\newblock Joint synthesis of safety certificate and safe control policy using
  constrained reinforcement learning.
\newblock In {\em Learning for Dynamics and Control Conference}, pages 97--109.
  PMLR, 2022.

\bibitem{chen2021lcss-saferl}
Hongyi Chen and Changliu Liu.
\newblock Safe and sample-efficient reinforcement learning for clustered
  dynamic environments.
\newblock {\em IEEE Control Systems Letters}, 6:1928--1933, 2021.

\bibitem{westenbroek2021ifac-saferl}
Tyler Westenbroek, Ayush Agrawal, Fernando Casta{\~n}eda, S~Shankar Sastry, and
  Koushil Sreenath.
\newblock Combining model-based design and model-free policy optimization to
  learn safe, stabilizing controllers.
\newblock {\em IFAC-PapersOnLine}, 54(5):19--24, 2021.

\bibitem{liu2023corl-safe}
Simin Liu, Changliu Liu, and John Dolan.
\newblock Safe control under input limits with neural control barrier
  functions.
\newblock In {\em Conference on Robot Learning}, pages 1970--1980. PMLR, 2023.

\bibitem{zhao2021fac-provableneuralcbf}
Hengjun Zhao, Xia Zeng, Taolue Chen, Zhiming Liu, and Jim Woodcock.
\newblock Learning safe neural network controllers with barrier certificates.
\newblock {\em Formal Aspects of Computing}, 33:437--455, 2021.

\bibitem{jin2020arxiv-provableneuralcbf}
Wanxin Jin, Zhaoran Wang, Zhuoran Yang, and Shaoshuai Mou.
\newblock Neural certificates for safe control policies.
\newblock {\em arXiv preprint arXiv:2006.08465}, 2020.

\bibitem{blanchini99automatica-set}
Franco Blanchini.
\newblock Set invariance in control.
\newblock {\em Automatica}, 35(11):1747--1767, 1999.

\bibitem{boyd04book-convex}
Stephen Boyd, Stephen~P Boyd, and Lieven Vandenberghe.
\newblock {\em Convex optimization}.
\newblock Cambridge university press, 2004.

\bibitem{jarvis03cdc-some}
Zachary Jarvis-Wloszek, Ryan Feeley, Weehong Tan, Kunpeng Sun, and Andrew
  Packard.
\newblock Some controls applications of sum of squares programming.
\newblock In {\em IEEE Conf. on Decision and Control (CDC)}, volume~5, pages
  4676--4681. IEEE, 2003.

\bibitem{yang22pami-certifiably}
Heng Yang and Luca Carlone.
\newblock Certifiably optimal outlier-robust geometric perception: Semidefinite
  relaxations and scalable global optimization.
\newblock {\em {IEEE} Trans. Pattern Anal. Machine Intell.}, 2022.

\bibitem{henrion05-detecting}
Didier Henrion and Jean-Bernard Lasserre.
\newblock Detecting global optimality and extracting solutions in gloptipoly.
\newblock {\em Positive polynomials in control}, 312:293--310, 2005.

\bibitem{yang22mp-inexact}
Heng Yang, Ling Liang, Luca Carlone, and Kim-Chuan Toh.
\newblock An inexact projected gradient method with rounding and lifting by
  nonlinear programming for solving rank-one semidefinite relaxation of
  polynomial optimization.
\newblock {\em Mathematical Programming}, pages 1--64, 2022.

\bibitem{nie14mp-optimality}
Jiawang Nie.
\newblock Optimality conditions and finite convergence of lasserre’s
  hierarchy.
\newblock {\em Mathematical programming}, 146:97--121, 2014.

\bibitem{bertsimas97book-lp}
Dimitris Bertsimas and John~N Tsitsiklis.
\newblock {\em Introduction to linear optimization}, volume~6.
\newblock Athena scientific Belmont, MA, 1997.

\end{thebibliography}

\end{document}